\documentclass[twoside]{aiml18}

\usepackage{aiml18macro}

\usepackage{graphicx}
\usepackage{amsmath}
\usepackage{amssymb}
\usepackage{amsrefs}
\usepackage{tikz}
\usetikzlibrary{matrix,arrows}
\usepackage{enumitem}

\newcommand{\eq}{\approx}

\newcommand{\m}[1]{{\bf {#1} }}

\newcommand{\si}{\ensuremath{\sigma}}

\DeclareMathOperator{\Con}{\mathrm{Con}}
\DeclareMathOperator{\KCon}{\mathrm{KCon}}
\newcommand{\cg}[1]{{\rm Cg}_{_{#1}}}

\newcommand{\mdl}[1]{\models_{#1}}
\newcommand{\lgc}[1]{\ensuremath{{\sf #1}}}
\newcommand{\De}{\Delta}

\newcommand{\Si}{\Sigma}
\newcommand{\The}{\Theta}

\newcommand{\F}{\m{F}}
\newcommand{\Tm}{\m{Tm}}

\newcommand{\lang}{\mathcal{L}}
\newcommand{\x}{\overline{x}}
\newcommand{\y}{\overline{y}}
\newcommand{\z}{\overline{z}}

\newcommand{\eps}{\ensuremath{\varepsilon}}
\newcommand{\N}{\mathbb{N}}
\newcommand{\V}{\mathcal{V}}

\def\lastname{Kowalski and Metcalfe}

\begin{document}

\begin{frontmatter}
  \title{Coherence in Modal Logic}
    \author{Tomasz Kowalski}
  \address{Department of Mathematics and Statistics\\La Trobe University, Melbourne, Australia\\T.Kowalski@latrobe.edu.au}
  \author{George Metcalfe}\footnote{The second author acknowledges support from Swiss National Science Foundation grant 200021$\_$165850 and the EU Horizon 2020 research and innovation programme under the Marie Sk{\l}odowska-Curie grant agreement No 689176.}
  \address{Mathematical Institute\\University of Bern, Switzerland \\ george.metcalfe@math.unibe.ch}
  
  \begin{abstract}
  A variety is said to be coherent if the finitely generated subalgebras of its finitely presented members are also finitely presented. In~\cite{KM18} it was shown that coherence forms a key ingredient of the uniform deductive interpolation property for equational consequence in a variety, and a general criterion was given for the failure of coherence (and hence uniform deductive interpolation) in varieties of algebras with a term-definable semilattice reduct. In this paper, a more general criterion is obtained and used to prove the failure of coherence and uniform deductive interpolation for a broad family of modal logics, including $\lgc{K}$, $\lgc{KT}$, $\lgc{K4}$, and $\lgc{S4}$.
  \end{abstract}

 \end{frontmatter}


\section{Introduction}\label{s:introduction}

A variety --- equivalently, a class of algebras defined by equations --- is said to be {\em coherent}\, if every finitely generated subalgebra of a finitely presented member of  the variety is again finitely presented. The notion of coherence originated in sheaf theory and has been studied quite widely in algebra, mostly in connection with groups, rings, modules, monoids, and lattices (see, e.g.,~\cites{ES70,CLL73,Sch83,Gou92}). It has also been considered from a more general model-theoretic perspective by Wheeler~\cites{Whe76,Whe78}, who proved, among other things, that coherence of a variety is implied by, and in conjunction with amalgamation and a further property implies, the existence of a model completion for its first-order theory. 

In~\cite{KM18} it was shown that there exists a close relationship between coherence and the logical notion of {\em uniform interpolation}. Following~\cite{vGMT17}, right uniform deductive interpolation is defined for equational consequence in a variety $\V$ as an extension of  deductive interpolation (studied in,~e.g.,~\cites{Jon65,Pig72,Mak77,Cze85,MMT14}) and  related to properties of compact congruences on the free algebras of $\V$. It was proved in~\cite{KM18}  that right uniform deductive interpolation is equivalent to the conjunction of deductive interpolation and coherence. Since  deductive interpolation is equivalent to amalgamation in the presence of the congruence extension property (see, e.g.,~\cite{MMT14}), coherence provides the key extra ingredient for an algebraic characterization of right uniform deductive interpolation. 

In this paper, we investigate coherence in {\em modal logic}: more precisely, we establish the failure of this property --- and hence also the failure of uniform deductive interpolation  --- for broad families of normal modal logics.

Following Pitts' seminal proof of uniform interpolation for intuitionistic propositional logic $\lgc{IPC}$~\cite{Pit92}, uniform interpolation was  established also for a number of important modal logics. In particular, semantic proofs making use of bisimulation quantifiers  were given by Visser in~\cite{Vis96} for the basic normal modal logic $\lgc{K}$ (see also~\cite{GZ02}), Grzegorczyk logic $\lgc{S4Grz}$, and G{\"o}del-L{\"o}b logic $\lgc{GL}$ (first proved by Shavrukov~\cite{Sha93}), and syntactic Pitts-style proofs were given by B{\'i}lkov{\'a} in~\cite{Bil07} for $\lgc{K}$ and $\lgc{KT}$. Relationships between uniform interpolation and bisimulation quantifiers for the modal $\mu$-calculus and other fixpoint modal logics have been studied in some depth in~\cites{dAH00,dAg05,MSV15}.  

Crucially for the topic of this paper, however, the above-mentioned proofs establish  an  ``implication-based'' uniform interpolation property, that for $\lgc{IPC}$, $\lgc{GL}$, and $\lgc{S4Grz}$ implies the ``consequence-based'' uniform deductive interpolation property studied in~\cites{vGMT17,KM18} (and implicitly in~\cite{GZ02}), but not for $\lgc{K}$ or $\lgc{KT}$. Indeed, it is proved in~\cite{KM18} that any coherent variety of modal algebras must have equationally definable principal congruences (EDPC) or, equivalently, that the corresponding modal logic must be weakly transitive, i.e., admit as a theorem $\boxdot^n x \to \boxdot^{n+1} x$ (where $\boxdot x := x \land \Box x$)  for some $n\in\N$. Since this is not the case for $\lgc{K}$ or $\lgc{KT}$, these logics are not coherent and do not admit uniform deductive interpolation. In the case of $\lgc{K}$, the failure of uniform deductive interpolation was already observed (at least implicitly) in~\cite{GLW06}, where it was shown that the description logic $\mathcal{ALC}$, a notational variant of multi-modal $\lgc{K}$,  does not have this property; note also that in~\cite{LW11} an algorithmic characterization was given of the formulas of this logic that do admit deductive uniform interpolants.

Failure of coherence for a family of non-weakly transitive modal logics (also substructural logics, bi-intuitionistic logic, and lattices) was established in~\cite{KM18} via a general criterion stating   that in a coherent variety $\V$ of algebras with a term-definable semilattice  reduct, any increasing and monotone term satisfying a certain fixpoint embedding condition in $\V$ admits a fixpoint obtained by iterating the term finitely many times. It was left open in~\cite{KM18}, however, as to whether this criterion could be used to show also the failure of coherence for  weakly transitive modal logics. Ghilardi and Zawadowski proved in~\cite{GZ95} that $\lgc{S4}$ does not admit uniform interpolation, and in~\cite{GZ02} gave a description of the axiomatic extensions of this logic admitting a model completion. In this paper, we obtain similar results using a  more general criterion for the failure of coherence that allows extra variables in terms satisfying the fixpoint embedding condition in some variety of modal algebras. 

In Section~\ref{s:interpolation}, we provide the required algebraic background and recall the relationship between uniform deductive interpolation and coherence established in~\cite{KM18}. In Section~\ref{s:criterion}, we then state and prove our new criterion for the failure of coherence. Finally, in Section~\ref{s:modal}, we apply the criterion with carefully chosen terms to obtain failures of coherence for broad families of non-weakly transitive and weakly transitive modal logics, including $\lgc{K}$, $\lgc{KT}$, $\lgc{K4}$, and $\lgc{S4}$. Failure of coherence implies failure of uniform deductive interpolation and  absence of a model completion, so our  results overlap with those obtained by Ghilardi and Zawadowski in~\cites{GZ95,GZ02}. However, our proofs are arguably simpler, since they require only finding a suitable term; moreover, the method is not confined to modal logics but can be applied to many other families of non-classical logics and varieties of algebras.


\section{Uniform Deductive Interpolation and Coherence} \label{s:interpolation}

In this section, we recall a general account of uniform deductive interpolation for varieties of algebras that was first presented in~\cite{vGMT17}, and a relationship, obtained in~\cite{KM18},  between right uniform deduction interpolation and the model-theoretic notion of coherence. For logics admitting a variety as an equivalent algebraic semantics (in particular, normal modal logics), these notions and results can also be easily translated into a logical setting. 

Let us fix an algebraic signature $\lang$ with at least one constant symbol\footnote{The restriction to one constant symbol is not necessary but simplifies certain aspects of the presentation.} and a variety of $\lang$-algebras $\V$. Given any set of variables $\x$, denote by $\Tm(\x)$ the \emph{$\lang$-term} {\em algebra over $\x$} and by $\F(\x)$ the {\em free algebra of $\V$ over $\x$}, which may be defined as the quotient of  $\Tm(\x)$ by the congruence $\Theta_\V$ defined by $s\,\Theta_\V\,t\, :\Leftrightarrow\, \V \models s \eq t$. We write $t(\x)$,  $\eps(\x)$, or $\Si(\x)$ to denote that the variables of an $\lang$-term $t$, $\lang$-equation $\eps$, or set of $\lang$-equations $\Si$, respectively, are included in $\x$. We also use $t$ to denote the corresponding $\Theta_\V$ equivalence class of $t$ in $\F(\x)$, relying on context to avoid ambiguity; similarly, we use $\eps$ and $\Si$ to denote corresponding pairs of elements and sets of pairs of elements, respectively, of $\F(\x)$. We assume throughout the paper that $\x$, $\y$, etc. denote disjoint sets, and that $\x,\y$ denotes their disjoint union.   

Consequence in $\V$ is defined as follows. For a set of $\lang$-equations $\Si \cup \{\eps\}$ containing exactly the variables in the set $\x$, let 
\[
\begin{array}{rcc}
\Si \mdl{\V} \eps & \,:\Longleftrightarrow\,  & \text{for every $\m{A} \in \V$ and
homomorphism $e 
\colon \Tm(\x) \to \m{A}$,}\\[.05in] 
& &  \Si \subseteq \ker(e) \ \Longrightarrow \ \eps \in \ker(e).
\end{array}
\]
For a set of $\lang$-equations $\Si \cup \De$, we write $\Si \mdl{\V} \De$ if $\Si \mdl{\V} \eps$ for all $\eps \in \De$. 

We say that $\V$ admits {\it deductive interpolation} if for any sets $\x,\y,\z$ and set of equations $\Si(\x,\y) \cup \{\eps(\y,\z)\}$ satisfying $\Si \mdl{\V} \eps$, there exists a set of equations $\Pi(\y)$ satisfying
\[
\Si \mdl{\V} \Pi \ \text{ and } \ \Pi \mdl{\V} \eps.
\]
Deductive interpolation and its relationships with other logical and algebraic properties have been investigated by many authors (see, e.g.,~\cites{Jon65,Pig72,Mak77,Cze85,MMT14}). In particular, it is well known (see, e.g.,~\cite{MMT14}) that if $\V$ has the amalgamation property, then it admits deductive interpolation and, conversely, if $\V$ admits deductive interpolation and has the congruence extension property, then it has the amalgamation property. 
 
Observe now (or consult~\cite{vGMT17}*{Proposition~2.10}) that $\V$ admits deductive interpolation if and only if for any finite sets $\x,\y$ and finite set of equations $\Si(\x,\y)$, there exists a set of equations $\Pi(\y)$ such that for any equation $\eps(\y,\z)$, 
\[ 
\Si \mdl{\V} \eps \iff \Pi \mdl{\V} \eps.
\]
Following~\cite{vGMT17}, we say that $\V$ admits {\em right uniform deductive interpolation} if $\Pi(\y)$ in the preceding condition is required to be finite.\footnote{Note that a similar definition can be given for left uniform deductive interpolation (see~\cite{vGMT17}), but in this paper, we focus only on failures of right uniform deductive interpolation, indeed only on cases where coherence fails.}  It is then easily proved (see~\cite{vGMT17}*{Proposition~3.5}) that $\V$ admits right uniform deductive interpolation if and only if

\begin{enumerate}
\item[\rm (i)]	 $\V$ admits deductive interpolation;
\item[\rm (ii)]	 for any finite sets $\x,\y$ and finite set of equations $\Si(\x,\y)$, there exists a set of equations $\Pi(\y)$ satisfying for any equation $\eps(\y)$, 
\[ 
\Si \mdl{\V} \eps \iff \Pi \mdl{\V} \eps.
\]
\end{enumerate}

\noindent
These notions may also be expressed in terms of congruences on the finitely generated free algebras of $\V$. Denote the congruence lattice of an algebra $\m{A}$ by $\Con \m{A}$ and the join-semilattice of compact (equivalently, finitely generated) congruences on $\m{A}$ by $\KCon \m{A}$, and write  $\cg{\m{A}}(S)$ to denote the congruence on $\m{A}$ generated by $S \subseteq A^2$. Recall also (see~\cite{MMT14}*{Lemma~2}) that for any sets of equations $\Si(\x), \De(\x)$, 
\[
\Si \models_\V \De \iff \cg{\F(\x)} (\De) \subseteq \cg{\F(\x)} (\Si).
\]
Observe next that the natural inclusion map  $i \colon \F(\y) \to  \F(\x,\y)$ ``lifts'' to the adjoint pair of maps
\[
\begin{array}{rcl}
i^* \colon \Con  \F(\y) \to \Con \F(\x,\y); & &
\Theta  \mapsto \cg{\F(\x,\y)}(i[\Theta])\\[.05in]
i^{-1} \colon \Con \F(\x,\y) \to  \Con \F(\y); & &
\Psi \mapsto i^{-1}[\Psi] = \Psi \cap \textup{F}(\y)^2.
\end{array}
\]
It is then straightforward to show that  $\V$  admits deductive interpolation if and only if for any finite sets $\x,\y,\z$, the following diagram commutes:
\begin{equation*}
\begin{tikzpicture}[baseline=(current  bounding  box.center)]
 \matrix (m) [matrix of math nodes, row sep=1.5em, column sep=1.5em]{
\Con \F(\x,\y) 		& 		& \Con  \F(\y) 		\\
 		&		& 		\\
\Con \F(\x,\y,\z) 	&     		& \Con \F(\y,\z) 	\\
};
\path
(m-1-1) edge[->] node[above] {$i^{-1}$} (m-1-3)
(m-1-1) edge[->] node[left] {$j^*$} (m-3-1)
(m-3-1) edge[->] node[below] {$k^{-1}$} (m-3-3)
(m-1-3) edge[->] node[right] {$l^*$} (m-3-3);
\end{tikzpicture}
\end{equation*}
where $i$, $j$, $k$, and $l$ are the inclusion maps between the corresponding finitely generated free algebras. 

Given any natural inclusion map  $i \colon \F(\y) \to  \F(\x,\y)$, the map $i^*$ will always restrict to a map $i^*|_{\KCon \F(\y)} \colon \KCon \F(\y) \to \KCon \F(\x,\y)$.  On the other hand, $i^{-1}$ restricts to $i^{-1}|_{\KCon \F(\x,\y)} \colon \KCon \F(\x,\y) \to \KCon \F(\y)$, yielding the right adjoint of $i^*|_{\KCon \F(\y)}$, if and only if $i$ preserves compact congruences; that is, if and only if for any compact congruence $\Psi$ on $\F(\x,\y)$, also $\Psi \cap \textup{F}(\y)^2$ is compact. It is not hard to see that this is exactly the case when condition (ii) for right uniform deductive interpolation is satisfied. Moreover, it was shown in~\cite{KM18} that this property is equivalent to the property of coherence, studied from a more general model-theoretic perspective by Wheeler in~\cites{Whe76,Whe78}. Recall that an algebra $\m{A} \in \V$ is  {\em finitely presented} (in $\V$) if it isomorphic to $\F(\x) / \The$ for some finite set $\x$ and $\The\in \KCon \F(\x)$.

\begin{theorem}[{\cite{KM18}*{Theorem~2.3}}]\label{p:characterization}
The following are equivalent:
\begin{enumerate}
\item[\rm (1)] 	For any finite sets $\x,\y$ and any finite set of equations $\Si(\x,\y)$, there exists a finite set of equations $\Pi(\y)$ such that for any equation $\eps(\y)$, 
  \[ 
\Si \mdl{\V} \eps \iff \Pi \mdl{\V} \eps.
\]
\item[\rm (2)] 	For any finite sets $\x,\y$ and any compact congruence $\The$ on $\F(\x,\y)$, the congruence $\The \cap F(\y)^2$  on $\F(\y)$ is compact.

\item[\rm (3)] 	$\V$ is {\em coherent}; that is, all finitely generated subalgebras of finitely presented members of $\V$ are finitely presented.\footnote{Note that, by our earlier assumption, coherence is defined here only for varieties in a signature $\lang$ that contains at least one constant symbol; this restriction is not essential, but allows for a neater presentation.} 

\end{enumerate}
\end{theorem}

It follows that $\V$ admits right uniform deductive interpolation if and only if $\V$ is coherent and admits deductive interpolation. As mentioned already  in the introduction, it was proved by Wheeler in~\cite{Whe76} that the coherence of $\V$ is implied by (and, in conjunction with amalgamation and another property, implies) the existence of a model completion for the first-order theory of $\V$. Hence establishing the failure of coherence for a variety yields also the failure of uniform deductive interpolation and lack of a model completion. Examples of coherent varieties include abelian groups and any locally finite variety (since in these varieties, finitely generated algebras are finitely presented), lattice-ordered abelian groups and MV-algebras (see~\cite{vGMT17}*{Example~3.7}), and Heyting algebras (the critical part of Pitts' theorem~\cite{Pit92}). Note, however, that the variety of groups is not coherent, since every finitely generated recursively presented group embeds into a finitely presented group~\cite{Hig61}.


\section{A General Criterion} \label{s:criterion}

In this section, we show that in any coherent variety $\V$ of algebras with a term-definable join-semilattice reduct, each term $t(x,\bar{u})$ that  is increasing and monotone, and satisfies a fixpoint embedding condition (with respect to~$x$), has a fixpoint obtained by iterating $t$ finitely many times in the first argument. This result provides a general criterion for establishing the failure of coherence and therefore also the failure of right uniform deductive interpolation and the lack of a model completion. It generalizes a similar result obtained in~\cite{KM18} by allowing extra variables $\bar{u}$ in the term $t(x,\bar{u})$; in Section~\ref{s:modal} we make good use of this flexibility in dealing with weakly transitive modal logics.

Given any term $t(x,\bar{u})$, define inductively
\[
t^0(x,\bar{u}) := x \quad \text{and} \quad t^{k+1}(x,\bar{u}) := t(t^{k}(x,\bar{u}),\bar{u}) \ \text{ for } k \in \N.
\]
For any algebra $\m{A}$ and term $s(x_1,\ldots,x_n)$, the term function $s^\m{A} \colon A^n \to A$ is defined inductively in the usual way; for convenience, however, we often omit the superscript $^\m{A}$ when referring to such functions. 

\begin{theorem}\label{t:general}
Let $\V$ be a coherent variety of $\lang$-algebras with a term-definable  join-semilattice reduct and a term $t(x,\bar{u})$ satisfying 
\[
\V \models x \le t(x,\bar{u}) \quad \text{ and } \quad \V \models x \le y\, \Rightarrow\,  t(x,\bar{u}) \le t(y,\bar{u}). 
\]  
Suppose also that $\V$ satisfies the following \emph{fixpoint embedding condition} with respect to $t(x,\bar{u})$:
\begin{enumerate}[leftmargin=*]
\item[\textup{(FE)}]
For any finitely generated $\m{A}\in \V$ and $a,\bar{b} \in A$, there exists an algebra $\m{B}\in\V$ such that $\m{A}$ is a subalgebra of $\m{B}$ and the join $\bigvee_{k \in \N} t^k(a,\bar{b})$ exists in $\m{B}$ and satisfies  
\[
 \bigvee_{k \in \N} t^k(a,\bar{b}) = t(\bigvee_{k \in \N} t^k(a,\bar{b}),\bar{b}).
\]
\end{enumerate}
Then $\V \models t^{n}(x,\bar{u}) \eq t^{n+1}(x,\bar{u})$ for some $n \in \N$.
\end{theorem}
\begin{proof}
Let  $\V$ and $t(x,\bar{u})$ be as in the statement of the theorem. Note first that the fact that $t$ is increasing and monotone easily implies that for any $n\in\N$,
\[
\V \models t^{n}(x,\bar{u}) \leq t^{n+1}(x,\bar{u}).
\]
To establish the converse inequality for some $n\in\N$, we define sets of equations
\[
\Si = \{y \le x, x \le z, x \eq t(x,\bar{u})\}
\quad
\mbox{and}
\quad
\Pi = \{ t^k (y,\bar{u}) \le z \mid k \in \N\}
\]
and prove that for any equation $\eps(y,z,\bar{u})$,
\[
\Si \mdl{\V} \eps(y,z,\bar{u}) \iff \Pi \mdl{\V} \eps(y,z,\bar{u}). \leqno{(\star)}
\]
For the right-to-left direction, it suffices to observe that $\Si \mdl{\V} t^k (y,\bar{u}) \le z$ for each $k \in \N$. For the left-to-right direction, suppose contrapositively that $\Pi \not \mdl{\V}\eps(y,z,\bar{u})$. Since only finitely many variables occur in $\Pi$, there exist a finitely generated $\m{A} \in \V$ and a homomorphism $e \colon \Tm(y,z,\bar{u}) \to \m{A}$ such that $\Pi \subseteq \ker(e)$, but $\eps
\not \in \ker(e)$. Let $a=e(y)$, $\bar{b} = e(\bar{u})$.  By assumption, $\m{A}$ is a subalgebra of some $\m{B}\in \V$ such that $\bigvee_{k \in \N} t^k(a,\bar{b})$ exists in $\m{B}$ and  
\[
 \bigvee_{k \in \N} t^k(a,\bar{b}) = t(\bigvee_{k \in \N}  t^k(a,\bar{b}),\bar{b}).
\]
Since $x$ does not appear in $\Pi \cup\{\eps\}$, we may extend $e$ to a homomorphism $e \colon \Tm(x,y,z,\bar{u}) \to \m{B}$ by defining 
\[
e(x) = \bigvee_{k \in \N} t^k(a,\bar{b}).
\]
But $t^k (a,\bar{b}) \le e(z)$ for each $k \in \N$, so clearly $e(y) \le e(x) \le e(z)$. Moreover, by the fixpoint embedding condition,
\[
e(x) = \bigvee_{k \in \N} t^k(a,\bar{b}) = t(\bigvee_{k \in \N}
t^k(a,\bar{b}),\bar{b}) = e(t(x,\bar{u})). 
\]
Hence $\Si \subseteq \ker(e)$ and we obtain $\Si \not \mdl{\V} \eps(y,z,\bar{u})$, completing the proof of $(\star)$.

Since $\V$ is coherent, there exists a finite set of equations $\De(y,z,\bar{u})$ such that for any equation $\eps(y,z,\bar{u})$, 
\[
\Si  \mdl{\V} \eps(y,z,\bar{u}) \iff \De \mdl{\V} \eps(y,z,\bar{u}).
\]
So $\Si \mdl{\V} \De$, and, by $(\star)$, also $\Pi \mdl{\V} \De$. Using the compactness of $\mdl{\V}$ (see~\cite{MMT14}) and the fact that $\De$ is finite, $\Pi' \mdl{\V} \De$ for some finite $\Pi' \subseteq \Pi$.  But also $\{t^{k+1} (y,\bar{u}) \le z\} \mdl{\V} t^{k} (y,\bar{u}) \le z$ for each $k\in\N$, and hence for some $n \in \N$,
\[
\{t^{n} (y,\bar{u}) \le z\} \mdl{\V} \De.
\] 
Since $\Si  \mdl{\V}  t^{n+1} (y,\bar{u}) \le z$, also $\De\mdl{\V}  t^{n+1} (y,\bar{u}) \le z$. Hence, combining these consequences, $\{t^{n} (y,\bar{u}) \le  z\}\mdl{\V}  t^{n+1} (y,\bar{u}) \le z$. Finally, substituting $z$ with $t^{n} (y,\bar{u})$ and $y$ with $x$, we obtain $\V \models t^{n+1} (x,\bar{u}) \le t^{n} (x,\bar{u})$.  
\end{proof}

A  less general  version of this theorem was used in~\cite{KM18} to establish the failure of coherence for broad families of varieties of Boolean algebras with operators  and residuated lattices, as well as the varieties of double-Heyting algebras (algebraic semantics for bi-intuitionistic logic) and lattices. 


\section{Modal Logics}\label{s:modal}

In this section, we apply the general criterion from Section~\ref{s:criterion} to a wide range of normal modal logics. Since the central notion of coherence is algebraic in nature and our Theorem~\ref{t:general} is formulated algebraically, let us for convenience call a normal modal logic $\lgc{L}$ {\em coherent} if the variety of modal algebras $\V_\lgc{L}$ providing an equivalent algebraic semantics for $\lgc{L}$ is coherent. Our definition of consequence in a variety $\V$, stated in Section~\ref{s:interpolation}, corresponds to \emph{global} consequence in modal logic, so our use of modal logic terminology here should always be taken in its global meaning. Moreover, we will only consider normal modal logics, so in this section, \emph{logic} is synonymous with \emph{normal modal logic}.

Recall  that a logic $\lgc{L}$ is {\em strongly Kripke complete} if for every set of formulas $\Gamma\cup\{\varphi\}$, whenever $\Gamma\not\vdash_{\lgc{L}}\varphi$, there exists a Kripke frame $\mathfrak{F}$ and a valuation $v$, such that $\mathfrak{F}\models_v \Gamma$ and $\mathfrak{F}\not\models_v\varphi$. Wolter showed in~\cite{Wol93} (see also~\cite{ZWC01}) that $\lgc{L}$ is strongly Kripke complete if and only if each at most countably generated $\m{A}\in\V_{\lgc{L}}$ embeds into a complex algebra $\mathfrak{G}^+\in\V_{\lgc{L}}$ of some Kripke frame $\mathfrak{G}$. A variety with this property is  said to be \emph{$\omega$-complex} (see~\cite{ZWC01}).

Let us also recall that a logic $\lgc{L}$ is \emph{$n$-transitive} for $n\in\N$ if $\vdash_\lgc{L} \boxdot^n x\rightarrow\boxdot^{n+1} x$, or, equivalently, if the reflexive closure $R$ of the accessibility relation in Kripke frames for $\lgc{L}$ satisfies $R^{n+1}\subseteq R^n$. A logic $\lgc{L}$ is called \emph{weakly transitive} if it is $n$-transitive for some $n\in\N$. Equivalently, $\lgc{L}$ is weakly transitive if $\V_{\lgc{L}}$ has \emph{equationally definable principal congruences (EDPC)} (see, e.g.,~\cite{KK06}). 


\subsection{Non-Weakly Transitive Logics}\label{s:nonweaklytransitivemodal}

In~\cite{KM18} it was shown (Theorem~5.2) that any canonical modal logic that is coherent must be weakly transitive. Hence coherence fails for all canonical modal logics that are not weakly transitive.\footnote{In fact, a more general result was established in~\cite{KM18}: if a variety $\V$ of Boolean Algebras with Operators (BAOs) is both canonical and coherent, then several of its term reducts, including $\V$ itself and various reducts to a single operator, have EDPC, so coherence fails for all canonical varieties of BAOs without EDPC.} The class of canonical non-weakly transitive modal logics is already rather large, including standard logics such as  $\lgc{K}$, $\lgc{KT}$, $\lgc{KD}$, $\lgc{KB}$, $\lgc{KTB}$, and many others. However, the next result shows that the assumption of canonicity used in~\cite{KM18} can be replaced by the weaker assumption that the modal logic is strongly Kripke complete.

\begin{theorem}\label{t:old-result}
Any coherent strongly Kripke complete modal logic is weakly transitive.
\end{theorem} 

\begin{proof}
Let $\lgc{L}$ be a coherent  strongly Kripke complete modal logic. Consider the term $t(x) = \Diamond x\vee x$. Clearly, $t$ is monotone and increasing; moreover, it defines an operator, so $t$ is completely additive in complex algebras. Let $\m{A}\in\V_{\lgc{L}}$ be finitely generated, and let $a\in A$.  Then $\m{A}$ is at most countable, so by the $\omega$-complexity of $\V_{\lgc{L}}$, it embeds into a complex algebra $\mathfrak{G}^+\in\V_{\lgc{L}}$ of some Kripke frame $\mathfrak{G}$ for $\lgc{L}$. By complete additivity, we have that in $\mathfrak{G}^+$,
$$
\bigvee_{k\in\N} t^k(a) = t\bigl(\bigvee_{k\in\N}t^k(a)\bigr).
$$
Hence $\V_\lgc{L}$ satisfies the fixpoint embedding condition (FE) with respect to $t$. But then by Theorem~\ref{t:general}, we have $\V_\lgc{L}\models t^{n+1}(x) \eq t^n(x)$ for some $n \in\N$, so $\lgc{L}$ is weakly transitive. 
\end{proof}

All canonical logics are strongly Kripke complete, but the converse does not hold; a counterexample can be obtained by applying Thomason simulation to the tense logic of the real line (see, e.g.,~\cite{ZWC01} for details). Hence Theorem~\ref{t:old-result} is slightly stronger than the results stated in~\cite{KM18}. Let us remark also that although the requirements of Theorem~\ref{t:general} are satisfied whenever each countable $\m{A}\in \V_\lgc{L}$ embeds into a direct product of finite algebras from $\V_\lgc{L}$, this property is too strong to produce interesting results. An embedding $\m{A}\hookrightarrow \prod_{i\in I}\m{B}_i$ can always be taken to be subdirect with subdirectly irreducible factors, which implies that each countable subdirectly irreducible algebra in $\V_\lgc{L}$ is finite. For modal algebras (but not in general), this further implies that $\V_\lgc{L}$ is locally finite and hence coherent.


\subsection{Weakly Transitive Logics}\label{s:weaklytransitivemodal}

Clearly, Theorem~\ref{t:old-result} cannot be used to show failures of coherence for weakly transitive logics. Also, its proof relies on the fact that $t$ defines an operator, so weakening the assumption of canonicity narrows the scope of applications of the method. Indeed, to extend our approach to weakly transitive logics, we require a term that does not define an operator. For example, to prove that the canonical logic $\lgc{K4}$ is not coherent, the term $\Diamond \Box x$ (which does not define an operator) can be used with Theorem~\ref{t:general}. This approach also works for some other weakly transitive logics, but not for the archetypal transitive logic $\lgc{S4}$. To establish the failure of coherence for $\lgc{S4}$ using Theorem~\ref{t:general}, a unary term $t$ will probably not suffice, and a \emph{positive} unary term will certainly not suffice, since the one-generated free positive interior algebra is finite (see~\cite{Mor18}).   

We therefore make use here of the ternary term  
\[
t(x,y,z) = \Box(y\lor\Box(z\lor x))\lor x.
\]
This term does not define an operator, and for any variety $\V$ of modal algebras,
\[
\V\models x \le t(x,y,z) \quad \text{ and } \quad \V\models x \le x'\, \Rightarrow\,  t(x,y,z) \le t(x',y,z).
\]
We will now show that for any modal logic $\lgc{L}$ whose Kripke frames include finite chains of arbitrary length, $\V_\lgc{L}\not\models t^{n}(x,y,z) \eq t^{n+1}(x,y,z)$ for all $n\in\N$. By a \emph{finite chain} we mean here any frame $\mathfrak{C}_n = (C,R)$ such that $|C| = n$ for some $n\in\N$, and the reflexive closure of $R$ is a total order on $C$. Note that, according to this definition, a finite chain is not uniquely determined by the number of its elements; indeed,  there are precisely $2^n$ finite chains with $n$ elements, one for each choice of (a subset of) reflexive points. We say then that a logic $\lgc{L}$ \emph{admits finite chains}, if for each $n\in\N$ there exists at least one $n$-element chain that is a frame for $\lgc{L}$.

\begin{lemma}\label{l:fin-chain}
Let $\lgc{L}$ be a modal logic admitting finite chains, and let $t(x,y,z)$ be as defined above. Then  $\V_\lgc{L}\not\models t^{n}(x,y,z) \eq t^{n+1}(x,y,z)$ for all $n\in\N$.
\end{lemma}

\begin{proof}
Fix $n\in\N$. By the construction of $t$, $\V_\lgc{L}\models t^{n}(x,y,z) \leq t^{n+1}(x,y,z)$, so we need to show that the converse inequality fails in $\V_\lgc{L}$. Following common practice, in what follows we write just $t^m$ for $t^m(x,y,z)$ ($m \in \mathbb{N}$).

Since $\lgc{L}$ admits finite chains, let $\mathfrak{C}_n = (C_n,R_n)$ be a $2n+1$-element chain that is a frame for $\lgc{L}$. Without loss of generality, $C_n = \{0,1,\dots,2n\}$, and $R_n$ extends the natural strict order on $C_n$ by making some points reflexive. Let $v$ be a valuation extending the map $v\colon \{x,y,z\}\to \mathcal{P}(\{0,1,\dots,2n\})$ given by  
\begin{eqnarray*}
v(x) = && v(z) = \{2i+1\mid 0 \le i\le n-1\},\\
v(y) = && \{2i\mid 0 \le i \le n\},
\end{eqnarray*}
as illustrated by the following diagram:
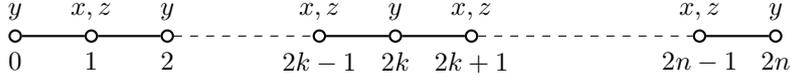
\[
\begin{tikzpicture}[dot/.style={inner sep=1.5pt,circle,draw,fill=white,thick}]
\node (v0) at (0,0) [dot,label=below:$0$,label=above:$y$] {};
\node (v1) at (1,0) [dot,label=below:$1$,label=above:{$x,z$}] {};
\node (v2) at (2,0) [dot,label=below:$2$,label=above:{$y$}] {};
\node (v3) at (4,0) [dot,label=below:$2k-1$,label=above:{$x,z$}] {};
\node (v4) at (5,0) [dot,label=below:$2k$,label=above:{$y$}] {};
\node (v5) at (6,0) [dot,label=below:$2k+1$,label=above:{$x,z$}] {};
\node (v6) at (9,0) [dot,label=below:$2n-1$,label=above:{$x,z$}] {};
\node (v7) at (10,0) [dot,label=below:$2n$,label=above:$y$] {};
\draw[thick] (v0) -- (v1) -- (v2);
\draw[dashed] (v2) -- (v3);
\draw[thick] (v3) -- (v4) -- (v5);
\draw[dashed] (v5) -- (v6);
\draw[thick] (v6) -- (v7);
\end{tikzpicture}
\]
Define $s_0 = \bot$ and $s_{m+1} = \Box(y\vee\Box(z\vee s_m))$. Then  $\V_\lgc{L}\models s_m \le t^m$ for all $m\in \N$. Just observe that, inductively, $\V_\lgc{L} \models s_0 \le x$, and since $t^{m+1} = \Box(y\vee\Box(z\vee t^m))\vee t^m$, we obtain $\V_\lgc{L} \models s_m \le t^m$ by the induction hypothesis, and then $\V_\lgc{L} \models s_{m+1} \le t^{m+1}$ as required. By the construction of $s_m$, it can only fail at a point $a$ (that is, $a \models_{\langle \mathfrak{C}_n,v\rangle} \lnot s_m$) if there is a path $a = a_0 \mathbin{R} a_1 \mathbin{R} \dots \mathbin{R} a_{2m-1} \mathbin{R} a_{2m}$ of (not necessarily distinct)  points $a_0,a_1,\dots,a_{2m}$ such that $\neg y$ holds at $a_i$ for all odd $i$, and $\neg z$ holds at $a_i$ for all even $i$. Inspecting $\mathfrak{C}_n$ we see that such a path does not exist for any $m > n$. So $s_m$ holds at every point of $\mathfrak{C}_n$, for any $m > n$. Since $\V_\lgc{L}\models s_{n+1} \le t^{n+1}$, also $t^{n+1}$ holds at every point of $\mathfrak{C}_n$.  

\medskip
\noindent
\emph{Claim.\/} For any $\ell\leq n$, we have that $2k\models_{\langle \mathfrak{C}_n,v\rangle} \neg t^\ell$ for all
$k\leq n-\ell$. 

\medskip
\noindent
\emph{Proof of Claim.\/}
Induction on $\ell$. For $\ell = 0$, it is immediate that for every $k\leq n$, we have $2k\models_{\langle  \mathfrak{C}_n,v\rangle}\neg x$. For $\ell+1\leq n$,  by the induction hypothesis, $t^\ell$ fails at $2k$ for all $k\leq n-\ell$, so we have the following:
\begin{eqnarray*}
2(n-\ell) &&\models_{\langle \mathfrak{C}_n,v\rangle} \neg t^\ell,\neg z\\
2(n-\ell)-1 &&\models_{\langle \mathfrak{C}_n,v\rangle} \neg y\\
2(n-\ell-1) &&\models_{\langle  \mathfrak{C}_n,v\rangle}\neg t^\ell,\neg z \\
   &&\vdots\\
2 &&\models_{\langle  \mathfrak{C}_n,v\rangle} \neg t^\ell,\neg z\\
1 &&\models_{\langle \mathfrak{C}_n,v\rangle} \neg y\\
0 &&\models_{\langle \mathfrak{C}_n,v\rangle} \neg t^\ell,\neg z.
\end{eqnarray*}
Since $t^{\ell+1} = \Box(y\vee\Box(z\vee t^\ell))\vee t^\ell$, we have that $t^{\ell+1}$ fails at $2(n-\ell-1)$ and at all even points  below it. This proves the claim. \qed
\medskip
\noindent

Finally, taking $\ell = n$ in the above claim, we obtain $0 \not\models_{\langle \mathfrak{C}_n,v\rangle} t^n$, and $0 \models_{\langle \mathfrak{C}_n,v\rangle}t^{n+1}$. Therefore,  $\mathfrak{C}_n^+\models t^{n+1}\not\leq t^n$ holds for the complex algebra $\mathfrak{C}_n^+$ of the frame $\mathfrak{C}_n$, and  $\V_\lgc{L}\not\models t^{n} \eq
t^{n+1}$, as required. 
\end{proof}

\noindent
Combining Theorem~\ref{t:general} with Lemma~\ref{l:fin-chain}, we obtain immediately the following sufficient condition for the failure of coherence in a modal logic.

\begin{theorem}\label{t:fin-chain}
Let $\lgc{L}$ be a modal logic admitting finite chains. If\/ $\V_\lgc{L}$ satisfies the  fixpoint embedding condition \textup{(FE)} with respect to the term $t(x,y,z)$ defined above, then $\V_\lgc{L}$ is not coherent.
\end{theorem}  
 
The main obstacle to applying Theorem~\ref{t:fin-chain} is the satisfaction of (FE) with respect to the term $t(x,y,z)$. Below, we explain why this condition is satisfied for all modal logics $\lgc{L}$ such that canonical extensions of countable members of $\V_\lgc{L}$ themselves belong to $\V_\lgc{L}$. If a variety $\V$ of Boolean Algebras with Operators (BAOs) possesses this property, it is said to be \emph{countably canonical}. It is an open question whether countable canonicity implies canonicity. 

We will assume basic knowledge about canonicity in algebraic form; in particular, we assume familiarity with the notion of a \emph{canonical extension} $\m{A}^{\!\si}$ of a BAO $\m{A}$. We refer the reader to~\cite{GH01} for the general theory of canonical extensions, and to~\cite{Jon94} for the necessary background on canonical extensions of BAOs. To keep the presentation smooth, we recall some terminology and facts, mostly from~\cite{Jon94}. Let $\m{A}$ be a BAO, and $t$ a term in the signature of $\m{A}$. If $t$ is a fundamental operation of $\m{A}$, the interpretation of $t$ in $\m{A}^{\!\sigma}$ is defined to be the canonical extension of $t^{\m{A}}$ as a map: put succinctly, $(t^{\m{A}})^\sigma = t^{\m{A}^{\!\sigma}}$. Although this equality is a definition for the fundamental operations, it is in general not preserved under composition. Hence, an arbitrary term-operation $t^{\m{A}}$ can be extended to an operation on $\m{A}^{\!\sigma}$ in two ways: as $(t^{\m{A}})^\sigma$, or as $t^{\m{A}^{\!\sigma}}$.  A term $t$ is called \emph{stable} if these two ways always coincide, that is, if for every BAO $\m{A}$ of appropriate signature, $(t^{\m{A}})^\sigma = t^{\m{A}^{\!\sigma}}$.  Terms defined by composing operators and lattice operations, or dual operators and lattice operations, are stable.  In particular, the term $t(x,y,z) = \Box(y\lor\Box(z\lor x))\lor x$  is stable. 

We will also need two lemmas spelling out certain fixpoint properties of canonical extensions. The first of these is a reformulation in terms of
canonical extensions of Esakia's Lemma, first proved in~\cite{Esa74}. (For a thorough treatment of Esakia's Lemma, and its connections to canonical extensions we refer the reader to~\cite{Geh14}.)

\begin{lemma}\label{l:gen-fix}
Let $\m{L}$ be a bounded lattice and let $f\colon L\to L$ be an order-preserving map. If $X\subseteq L$ is upward directed and closed under $f$, and $f$ is increasing on $X$, then  $f^\si(\bigvee X) = \bigvee X$ in $\m{L}^\si$. 
\end{lemma}

The second is a reformulation of Lemma~4.10, from~\cite{KM18}, where it was shown to hold, in a slightly different form, for a more general class of algebras. Here we state it for BAOs.

\begin{lemma}\label{l:spec-fix}
Let $\V$ be a variety of Boolean Algebras with Operators and let $t(x,\bar{u})$ be a stable term such that
\[
\V \models x \le t(x,\bar{u}) \quad \text{and} \quad \V \models x \le y\,  \Rightarrow\,  t(x,\bar{u}) \le t(y,\bar{u}).
\] 
Let $\m{A}\in\V$, $a,\bar{b}\in A$, and $X = \{(t^k)^{\m{A}}(a,\bar{b})\mid k\in \N\}$.  Then $\bigvee\! X$ exists in $\m{A}^{\!\si}$ and $t^{\m{A}^\si}(\bigvee\! X, \bar{b}) = \bigvee\! X$.  
\end{lemma}

\begin{proof}
Let $f\colon A \to A$ be the map defined by $f(x) = t^{\m{A}}(x,\bar{b})$. Then $f$ is order-preserving, $X$ is upward directed, closed under $f$,
and $f$ is increasing on $X$. Hence, by Lemma~\ref{l:gen-fix}, we have $(t^{\m{A}})^\sigma(\bigvee X,\bar{b}) = \bigvee X$ in $\m{A}^{\!\sigma}$.
But $t$ is stable, so $t^{\m{A}^\si}(\bigvee\! X, \bar{b}) = \bigvee\! X$  in
$\m{A}^{\!\si}$, as required.   
\end{proof}

\begin{theorem}\label{t:main}
Let $\lgc{L}$ be a modal logic admitting finite chains, and such that  $\V_\lgc{L}$ is countably canonical. Then
\begin{enumerate}
\item[\rm (a)] $\V_\lgc{L}$ is not coherent;
\item[\rm (b)] $\V_\lgc{L}$ does not admit right uniform deductive interpolation;
\item[\rm (c)] the first-order theory of\/ $\V_\lgc{L}$ does not have a model completion.
\end{enumerate}  
\end{theorem}   

\begin{proof}
By Theorem~\ref{t:fin-chain}, to prove (a), it suffices to show that $\V_\lgc{L}$ satisfies the fixpoint embedding condition (FE)  with respect to $t(x,y,z)$. Let $\m{A}\in\V_\lgc{L}$ be finitely generated. Then $\m{A}$ is at most countable, so $\m{A}^{\!\si}\in\V_\lgc{L}$ because $\V_\lgc{L}$ is countably canonical. Since $t(x,y,z)$ is stable, Lemma~\ref{l:spec-fix} guarantees that $t(\bigvee_{k\in\N}t^k(a,b,c))  = \bigvee_{k\in\N}t^k(a,b,c)$ in $\m{A}^{\!\si}$ for all $a,b,c \in A$. So $\V_\lgc{L}$ satisfies (FE) as required. The statements  (b) and (c) then follow using Theorem~\ref{p:characterization} and the fact, proved in~\cite{Whe78}, that if the first-order  theory of a variety has a model completion, then the variety is coherent.
\end{proof}  

We can also formulate a positive version of this theorem for coherent logics.

\begin{corollary}
Let $\lgc{L}$ be any coherent modal logic for which $\V_\lgc{L}$ is countably canonical. Then for any stable term $t(x,\bar{u})$ where $x$ occurs only positively, there exists $n \in \N$ such that $\vdash_\lgc{L} t_+^{n}(x,\bar{u}) \leftrightarrow t_+^{n+1}(x,\bar{u})$, where $t_+(x,\bar{u}) := x \lor t(x,\bar{u})$.
 \end{corollary}
 
 Theorem~\ref{t:main} implies the failure of coherence for a broad family of modal logics, including $\lgc{K}$, $\lgc{KT}$, $\lgc{K4}$, 
$\lgc{K4M}$ (McKinsey's logic), $\lgc{S4}$, and $\lgc{S4.3}$.  However, failure of coherence \emph{does not} follow from this theorem for logics such as $\lgc{GL}$, $\lgc{S4Grz}$, and $\lgc{S4.3Grz}$ that admit finite chains but are not canonical. It is known  that $\lgc{GL}$ (see~\cite{Sha93}) and $\lgc{S4Grz}$ (see~\cite{Vis96}) have uniform interpolation and are therefore coherent; indeed, it was proved in~\cite{GZ02} that the first-order theories of the varieties corresponding to these logics admit model completions. It was also proved in~\cite{GZ02} that the first-order theory of the variety corresponding to $\lgc{S4.3Grz}$ does not have a model completion, but it is not clear (at least to the present authors) if the proof given there also establishes the failure of coherence.  A general negative result proved in~\cite{GZ02} demonstrates that for any logic $\lgc{L}$ extending $\lgc{K4}$ that has the finite model property and admits all finite reflexive chains and the two-element cluster, $\V_{\lgc{L}}$ does not have a model completion. An analysis of this proof reveals that such logics are also not coherent. Clearly, there is a large overlap between this negative result and Theorem~\ref{t:main} (but intriguingly no inclusion either way), although our proofs are arguably simpler, since they require only finding a suitable term. Let us emphasize also that our method applies not only to modal logics but also to many other families of non-classical logics and varieties of algebras.

Let us remark finally that it would be rather easy to state the results of this section in a more general way for arbitrary varieties of BAOs. In any such variety, a great range of unary operators can be term-defined. Let $\boxtimes$ be one of them. Then we can define $t(x,y,z) = \boxtimes(y\lor\boxtimes(z\lor x))\lor x$, reformulate the condition of admitting finite chains algebraically, and state suitable analogues of Theorem~\ref{t:fin-chain} and~\ref{t:main}. Such an approach, however, would lack the simplicity and elegance of the presentation given here just for modal logics.



\bibliographystyle{aiml18}

\end{document}